\documentclass[11pt]{amsart}

\usepackage{amsmath}
\usepackage{amssymb}
\usepackage{graphicx}

\newtheorem{theorem}{Theorem}[section]
\newtheorem{corollary}[theorem]{Corollary}
\newtheorem{lemma}[theorem]{Lemma}
\newtheorem{claim}{Claim}

\theoremstyle{definition}
\newtheorem{definition}[theorem]{Definition}
\newtheorem{example}[theorem]{Example}
\newtheorem{remark}[theorem]{Remark}

\numberwithin{equation}{section}

\title[Convergence theorem for solving generalized mixed equilibrium]{Convergence theorem for solving generalized mixed equilibrium problem and finding fixed point of a weak Bregman relatively nonexpansive mapping in Banach spaces}

\author[V. Darvish]{V. Darvish}

\address[Vahid Darvish]{School of Mathematics and Statistics, Nanjing University of Information Science and Technology,
Nanjing 210044, China}
\email{{\tt vahid.darvish@mail.com}}

\author[K. Jantakarn]{K. Jantakarn}

\address[Kittisak Jantakarn]{Department of Mathematics
Faculty of Science
Naresuan University
Phitsanulok 65000, Thailand}
\email{{\tt kittisakj61@nu.ac.th}}

\author[A. Kaewcharoen]{A. Kaewcharoen}

\address[Anchalee Kaewcharoen]{Department of Mathematics
Faculty of Science
Naresuan University
Phitsanulok 65000, Thailand}
\email{{\tt anchaleeka@nu.ac.th}}

\author[N. Biranvand]{N. Biranvand}

\address[Nader Biranvand]{Faculty of Sciences, Imam Ali University, Tehran, Iran.}
\email{{\tt nabiranvand@gmail.com}}



\begin{document}

\begin{abstract}
In this paper, we study a new iterative method for finding the fixed point of a weak Bregman relatively nonexpansive mapping and the set of solutions of generalized mixed equilibrium problems in Banach spaces.
\end{abstract}

\maketitle

\section{Introduction}
Let $E$ be a real reflexive Banach space and $C$ a nonempty, closed and convex subset of $E$ and  $E^{*}$ be the dual space of $E$ and $f:E\to (-\infty,+\infty]$ be a proper, lower semi-continuous and convex function. We denote by $\text{dom} f$, the domain of $f$, that is the set $\{x\in E : f(x)<+\infty\}$. Let $x\in \text{int}(\text{dom} f)$, the subdifferential  of $f$ at $x$ is the convex set defined by 
\begin{equation*}
\partial f(x)=\{x^{*}\in E^{*} : f(x)+\langle x^{*},y-x\rangle \leq f(y), \forall y\in E\},
\end{equation*}
where the Fenchel conjugate of $f$ is the function $f^{*}: E^{*}\to (-\infty,+\infty]$ defined by 
$$f^{*}(x^{*})=\sup \{\langle x^{*},x\rangle -f(x): x\in E\}.$$

Equilibrium problems which were introduced by Blum and Oettli \cite{blu} and Noor and Oettli \cite{asl} in 1994 have had a great
impact and influence in the development of several branches of pure and applied sciences. It has been shown that the
equilibrium problem theory provides a novel and unified treatment of a wide class of problems which arise in economics,
finance, image reconstruction, ecology, transportation, network, elasticity and optimization. It has been shown (\cite{blu},\cite{asl}) that
equilibrium problems include variational inequalities, fixed point, Nash equilibrium and game theory as special cases. Hence
collectively, equilibrium problems cover a vast range of applications. Due to the nature of the equilibrium problems, it is
not possible to extend the projection and its variant forms for solving equilibrium problems. To overcome this drawback,
one usually uses the auxiliary principle technique. The main and basic idea in this technique is to consider an auxiliary
equilibrium problem related to the original problem and then show that the solution of the auxiliary problems is a solution
of the original problem. This technique has been used to suggest and analyze a number of iterative methods for solving
various classes of equilibrium problems and variational inequalities, see \cite{asl2} and the references therein.
Related to the equilibrium problems, we also have the problem of finding the fixed points of the nonexpansive mappings,
which is the subject of current interest in functional analysis. It is natural to construct a unified approach for these problems.
In this direction, several authors have introduced some iterative schemes for finding a common element of a set of the
solutions of the equilibrium problems and a set of the fixed points of finitely many nonexpansive mappings. 

Let $\Theta :\ C\times C\longrightarrow \Bbb R$ be a bifunction, where $\Bbb R$ is the set of real numbers, $\Psi :\ X\longrightarrow X^*$ be a nonlinear operator and $\varphi :\ C\longrightarrow \Bbb R$ be a real valued function. The generalized mixed equilibrium problem is to find an element $x\in C$ such that
\begin{equation}\label{eq11}
\Theta (x,y)+\langle \Psi x,y-x\rangle +\varphi (y)\geq \varphi (x)\ \forall y\in C.
\end{equation}
The set of solutions of the problem \eqref{eq11} is denoted by $GMEP(\Theta,\varphi ,\Psi )$, that is,
$$GMEP (\Theta,\varphi ,\Psi  )=\{x\in C: \ \Theta (x,y)+\langle \Psi x,y-x\rangle +\varphi (y)\geq \varphi (x)\ \forall y\in C\}.$$
Let $\Phi _i$, $i=1,2,...,N$ be $N$ bifunctions from $C\times C$ to $\Bbb R$,   $\varphi _i$, $i=1,2,...,N$ be $N$ real value functions from $C$ to $\Bbb R$ and   $\Psi _i$, $i=1,2,...,N$ be $N$ operators form $X$ to $X^*$. Solving a system of generalized mixed equilibrium problems means finding an element $x\in C$ such that
$x\in\cap _{i=1}^NGMEP(\Theta _i,\varphi _i,\Psi _i).$
In particular, if $\Psi =0$,   problem \eqref{eq11} is reduced to the following  mixed equilibrium problem, which is to find an element $x\in C$ such that
\begin{equation}\label{eq12}
\Theta (x,y)+\varphi (y)\geq \varphi (x)\ \forall y\in C.
\end{equation}
We denote by $MEP(\Theta )$ the set of solutions of   problem \eqref{eq12}.
If $\varphi =0$,   problem \eqref{eq11} is reduced to the following generalized equilibrium problem, which is to find an element $x\in C$ such that
\begin{equation}\label{eq13}
\Theta (x,y)+\langle \Psi x,y-x\rangle\geq 0\ \forall y\in C.
\end{equation}
The set of solutions of   problem \eqref{eq13} is denoted by $GEP(\Theta ,\Psi )$.
If $\Theta =0$,   problem \eqref{eq11} is reduced to the following mixed variational inequality of Browder type, which is to find an element $x\in C$ such that
\begin{equation}\label{eq14}
\langle \Psi x,y-x\rangle+\varphi (y)\geq \varphi (x)\ \forall y\in C.
\end{equation}
The set of solutions of the problem \eqref{eq14} is denoted by $MVI(C,\varphi ,\Psi )$.
If $\varphi =0$ and $\Psi =0$,  problem \eqref{eq11} is reduced to the following well known equilibrium problem,  which is to find an element $x\in C$ such that
\begin{equation}\label{eq15}
\Theta (x,y)\geq 0\ \forall y\in C.
\end{equation}
The set of solutions of   problem \eqref{eq15} is denoted by $EP(\Theta )$.

In \cite{rei}, Reich and Sabach proposed an algorithm for finding a common fixed point of finitely many Bregman strongly nonexpansive mappings $T_{i}:C\to C (i=1,2,\ldots, N)$ satisfying $\cap_{i=1}^{N}F(T_{i})\neq \emptyset$ in a reflexive Banach space $E$ as follows:
\begin{eqnarray*}
x_{0}&\in & E, \text{chosen arbitrarily,}\\
y_{n}^{i}&=&T_{i}(x_{n}+e_{n}^{i}),\\
C_{n}^{i}&=&\{z\in E : D_{f}(z,y_{n}^{i})\leq D_{f}(z,x_{n}+e_{n}^{i})\},\\
C_{n}&=&\cap_{i=1}^{N}C_{n}^{i},\\
Q_{n}^{i}&=&\{z\in E : \langle \nabla f(x_{0})-\nabla f(x_{n}), z-x_{n}\rangle\leq 0\},\\
x_{n+1}&=&proj_{C_{n}\cap Q_{n}}^{f}(x_{0}), \ \ \forall n\geq0,
\end{eqnarray*}
and
\begin{eqnarray*}
x_{0}&\in & E,\\
C_{0}^{i}&=&E, i=1,2,\ldots,N,\\
y_{n}^{i}&=&T_{i}(\nu_{n}+e_{n}^{i}),\\
C_{n+1}^{i}&=&\{z\in C_{n}^{i} : D_{f}(z,y_{n}^{i})\leq D_{f}(z,x_{n}+e_{n}^{i})\},\\
C_{n+1}&=&\cap_{i=1}^{N}C_{n+1}^{i},\\
x_{n+1}&=&proj_{C_{n+1}}(x_{0}), \ \ \forall n\geq0,
\end{eqnarray*}
where $proj_{C}^{f}$ is the Bregman projection with respect to $f$ from E onto  a closed and convex subset $C$ of $E$. They proved that the sequence $\{x_{n}\}$ converges strongly to a common fixed point of $\{T_{i}\}_{i=1}^{N}$.

The authors of \cite{aga} introduced the following algorithm:
\begin{align}
&x_{0}=x\in C \ \ \ \ \  \text{chosen arbitrarily},\nonumber\\
&z_{n}=\nabla f^{*}(\beta_{n}\nabla f(T(x_{n}))+(1-\beta_{n})\nabla f(x_{n})), \nonumber\\
&y_{n}=\nabla f^{*}(\alpha_{n} \nabla f(x_{0})+(1-\alpha_{n})\nabla f(z_{n})),\nonumber\\
&u_{n}=Res_{H}^{f}(y_{n}),\nonumber\\
&C_{n}=\{z\in C_{n-1}\cap Q_{n-1}: D_{f}(z,u_{n})  \nonumber \\
& \hspace{5 cm}\leq \alpha_{n}D_{f}(z,x_{0})+(1-\alpha_{n})D_{f}(z,x_{n})\}, \nonumber\\
&Q_{n}=\{z\in C_{n-1}\cap Q_{n-1}: \langle\nabla f(x_{0})-\nabla f(x_{n}),z-x_{n}\rangle\leq 0\},\nonumber\\
&x_{n+1}= proj^{f}_{C_{n}\cap Q_{n}}x_{0}, \  \forall n\geq 0, \label{e1.6}
\end{align}
where $H$ is an equilibrium bifunction and $T$ is a weak Bregman relatively nonexpansive mapping. They proved the above sequence  converges strongly to the point $proj_{F(T)\cap EP(H)}x_{0}$.

In this paper, motivated by the above algorithms, we study  the following iterative scheme:
\begin{align}
&z_{n}=\nabla f^{*}(\beta_{n}\nabla f(T(x_{n}))+(1-\beta_{n})\nabla f(x_{n})), \nonumber\\
&y_{n}=\nabla f^{*}(\alpha_{n} \nabla f(x_{0})+(1-\alpha_{n})\nabla f(z_{n})),\nonumber\\
&u_{n}=Res_{\Theta,\varphi,\Psi}^{f}(y_{n}),\nonumber\\
&C_{n+1}=\{z\in C_{n}: D_{f}(z,u_{n})\leq \alpha_{n}D_{f}(z,x_{0})+(1-\alpha_{n})D_{f}(z,x_{n})\}, \nonumber\\
&Q_{n+1}=\{z\in Q_{n}: \langle\nabla f(x_{0})-\nabla f(x_{n}),z-x_{n}\rangle\leq 0\},\nonumber\\
&x_{n+1}= proj^{f}_{C_{n+1}\cap Q_{n+1}}x_{0}, \ \forall n\geq 0,\label{eqw}
\end{align}
where $T$ is a weak Bregman relatively nonexpansive mapping, $\varphi :C\to \mathbb{R}$ is real-valued function, $\Psi: C\to E^{*}$ is continuous monotone mapping, $\Theta: C\times C\to \mathbb{R}$ is equilibrium bifunction. We will prove that the sequence $\{x_{n}\}$ defined in (\ref{eqw}) converges strongly to the point $proj_{F(T)\cap GMEP(\Theta)}x_{0}$. 
\section{Preliminaries}
For any $x\in \text{int}(\text{dom} f)$, the right-hand derivative of $f$ at $x$ in the derivation $y\in E$ is defined by
$$f^{'}(x,y):=\lim _{t\searrow0} \frac{f(x+ty)-f(x)}{t}.$$
The function $f$ is called G\^{a}teaux differentiable at $x$ if $\lim_{t\searrow0} \frac{f(x+ty)-f(x)}{t}$ exists for all $y\in E$. In this case, $f^{'}(x,y)$ coincides with $\nabla f(x)$, the value of the gradient ($\nabla f)$ of $f$ at $x$. The function $f$ is called G\^{a}teaux differentiable if it is G\^{a}teaux differentiable for any $x\in \text{int}(\text{dom} f)$ and $f$ is called Fr\'{e}chet differentiable at $x$ if this limit is attain uniformly for all $y$ which satisfies $\|y\|=1$. The function $f$ is uniformly Fr\'{e}chet differentiable on a subset $C$ of $E$ if the limit is attained uniformly for any $x\in C$ and $\|y\|=1$. It is known that if $f$ is G\^{a}teaux differentiable (resp. Fr\'{e}chet differentiable) on $\text{int}(\text{dom} f)$, then $f$ is continuous and its G\^{a}teaux derivative $\nabla f$ is norm-to-weak$^*$ continuous (resp. continuous) on $\text{int} (\text{dom}f)$ (see \cite{bon}).

\begin{definition}\cite{bre}
Let $f: E\to (-\infty,+\infty]$ be a G\^{a}teaux differentiable function. The function $D_{f}: \text{dom} f\times \text{int}(\text{dom} f)\to [0,+\infty)$ defined as follows:
\begin{equation}\label{1}
D_{f}(x,y):=f(x)-f(y)-\langle \nabla f(y),x-y\rangle
\end{equation}
is called the Bregman distance with respect to $f$.
\end{definition}

\begin{remark}\cite{rei15}
The Bregman distance has the following properties:
\begin{enumerate}
\item
the three-point identity, for any $x\in \text{dom} f$ and $y,z\in \text{int}(\text{dom} f)$,
$$D_{f}(x,y)+D_{f}(y,z)-D_{f}(x,z)=\langle \nabla f(z)-\nabla f(y),x-y\rangle;$$
\item
the four-point identity, for any $y,w\in \text{dom}f$ and $x,z\in\text{int}(\text{dom}f)$,
$$D_{f}(y,x)-D_{f}(y,z)-D_{f}(w,x)+D_{f}(w,z)=\langle \nabla f(z)-\nabla f(x),y-w\rangle.$$
\end{enumerate}
\end{remark}

The Legendre function $f:E\to (-\infty,+\infty]$ is defined in \cite{bau}. It is well known that in reflexive spaces, $f$ is Legendre function if and only if it satisfies the following conditions:

($L_{1}$) The interior of the domain of $f$, $\text{int}(\text{dom} f)$, is nonempty, $f$ is G\^{a}teaux differentiable on $\text{int}(\text{dom} f)$ and $\text{dom} f=\text{int}( \text{dom} f)$;

($L_{2}$) The interior of the domain of $f^{*}$, $\text{int}( \text{dom} f^{*})$, is nonempty, $f^{*}$ is G\^{a}teaux differentiable on $\text{int}(\text {dom} f^{*})$ and $\text{dom} f^{*}= \text{int}( \text{dom} f^{*})$.

\noindent Since $E$ is reflexive, we know that $(\partial f)^{-1}=\partial f^{*}$ (see \cite{bon}). This , with ($L_{1}$) and ($L_{2}$), imply the following equalities: 
$$ \nabla f=(\nabla f^{*})^{-1},  \ \ \ \text{ran} \nabla f=\text{dom} \nabla f^{*}=\text{int}(\text{dom} f^{*})$$
 and $$\text {ran} \nabla f^{*}=\text{dom}(\nabla f)=\text {int}(\text{dom} f),$$ where $\text{ran}\nabla f$ denotes the range of $\nabla f$. 

When the subdifferential of $f$ is single-valued, it coincides with the gradient $\partial f=\nabla f$, \cite{phe}. By Bauschke
et al \cite{bau} the conditions ($L_{1}$) and ($L_{2}$) also yields that the function $f$ and $f^{*}$ are strictly convex on the interior of their respective domains.\\
If $E$ is a smooth and strictly convex Banach space, then an important and interesting Legendre function is $f(x):=\frac{1}{p}\|x\|^{p} (1<p<\infty).$ In this case the gradient $\nabla f$ of $f$ coincides with the generalized duality mapping of $E$, i.e., $\nabla f=J_{p} (1<p<\infty).$ In particular, $\nabla f=I$, the identity mapping in Hilbert spaces. From now on we assume that the convex function $f:E\to (-\infty, \infty]$ is Legendre. In connection with Legendre functions, see also the recent paper \cite{ree}.
\begin{definition}\cite{bre}
Let $f:E\to (-\infty,+\infty]$ be a convex and G\^{a}teaux differentiable function. The Bregman projection of $x\in \text{int}(\text{dom} f)$ onto the nonempty, closed and convex subset $C\subset \text{dom} f$ is the necessary unique vector $proj_{C}^{f}(x)\in C$ satisfying 
$$D_{f}(proj_{C}^{f}(x),x)=\inf\{D_{f}(y,x) : y\in C\}.$$
\end{definition}
\begin{remark}
If $E$ is a smooth and strictly convex Banach space and $f(x)=\|x\|^{2}$ for all $x\in E$, then we have that $\nabla f(x)=2Jx$ for all $x\in E$, where $J$ is the normalized duality mapping from $E$ in to $2^{E^{*}}$, and hence $D_{f}(x,y)$ reduced to $\phi(x,y)=\|x\|^{2}-2\langle Jy, x\rangle +\|y\|^{2}$, for all $x,y\in E$, which is the Lyapunov function introduced by Alber \cite{alb} and Bregman projection $P_{C}^{f}(x)$ reduces to the generalized projection $\Pi_{C}(x)$ which is defined by
$$\phi(\Pi_{C}(x),x)=\min_{ y\in C} \phi(y,x).$$
If $E=H$, a Hilbert space, $J$ is  the identity mapping and hence Bregman projection $P_{C}^{f}(x)$ reduced to the metric projection of $H$ onto $C$, $P_{C}(x)$.
\end{remark}

\begin{definition}\cite{but4,but2}
Let $f:E\to(-\infty,+\infty]$ be a convex and G\^{a}teaux differentiable function. $f$ is called:
\begin{enumerate}
\item
\textit{totally convex} at $x\in \text{int}(\text{dom} f)$ if its modulus of total convexity at $x$, that is, the function $\nu_{f}:\text{int}(\text{dom} f)\times[0,+\infty)\to[0,+\infty)$ defined by 
$$\nu_{f}(x,t):=\inf\{D_{f}(y,x): y\in \text{dom}f, \|y-x\|=t\},$$
is positive whenever $t>0$;
\item
totally convex if it is totally convex at every point $x\in \text{int}(\text{dom} f)$;
\item
totally convex on bounded sets if $\nu_{f}(B,t)$ is positive for any nonempty bounded subset $B$ of $E$ and $t>0$, where the modulus of total convexity of the function $f$ on the set $B$ is the function $\nu_{f}:\text{int}(\text{dom} f)\times [0,+\infty)\to [0,+\infty)$ defined by 
$$\nu_{f}(B,t):=\inf\{\nu_{f}(x,t): x\in B\cap \text{dom} f\}.$$
\end{enumerate}
\end{definition}
The set $lev_{\leq}^{f}(r)=\{x\in E: f (x)\leq r\}$ for some $r\in\mathbb{R}$ is called a sublevel of $f$.
\begin{definition}\cite{but2,rei}\label{deff}
The function $f:\ E\to(-\infty,+\infty]$ is called;
\begin{enumerate}
\item
\textit{cofinite} if $\text{dom} f^{*}=E^{*}$;
\item
\textit{coercive} \cite{hir} if the sublevel set of $f$ is bounded; equivalently,
$$\lim_{\|x\|\to+\infty}f(x)=+\infty;$$
\item
\textit{strongly coercive} if $\lim_{\|x\|\to+\infty}\frac{f(x)}{\|x\|}=+\infty$;
\item
\textit{sequentially consistent} if for any two sequences $\{x_{n}\}$ and $\{y_{n}\}$ in $E$ such that $\{x_{n}\}$ is bounded,
$$\lim_{n\to\infty} D_{f}(y_{n},x_{n})=0\Rightarrow \lim_{n\to\infty}\|y_{n}-x_{n}\|=0.$$
\end{enumerate}
\end{definition}
\begin{lemma}\cite{but}\label{lem6}
The function $f$ is totally convex on bounded subsets if and only if it is sequentially consistent.
\end{lemma}

\begin{lemma}\cite[Proposition 2.3]{rei}
If $f:E\to(-\infty,+\infty]$ is Fr\'{e}chet differentiable and totally convex, then $f$ is cofinite.
\end{lemma}
\begin{lemma}\cite{but}\label{2.2.f}
Let $f:E\to(-\infty,+\infty]$ be a convex function whose domain contains at least two points.Then the following statements hold:
\begin{enumerate}
\item
$f$ is sequentially consistent if and only if it is totally convex on bounded sets;

\item
If $f$ is lower semicontinuous, then $f$ is sequentially consistent if and only if it is uniformly convex on bounded sets;

\item
If $f$ is uniformly strictly convex on bounded sets, then it is sequentially consistent and the converse implication holds when $f$ is lower semicontinuous, Fr\'{e}chet differentiable on its domain and Fr\'{e}chet derivative $\nabla f$ is uniformly continuous on bounded sets.

\end{enumerate}

\end{lemma}
\begin{lemma}\cite[Proposition 2.1]{rei4}\label{lem7}
Let $f:E\to\mathbb{R}$ be uniformly Fr\'{e}chet differentiable and bounded on bounded subsets of $E$. Then $\nabla f$ is uniformly continuous on bounded subsets of $E$ from the strong topology of $E$ to the strong topology of $E^{*}$.
\end{lemma}
\begin{lemma}\cite[Lemma 3.1]{rei}\label{aust}
Let $f:E\to \mathbb{R}$ be a G\^{a}teaux differentiable and totally convex function. If $x_{0}\in E$ and the sequence $\{D_{f}(x_{n},x_{0})\}$ is bounded, then the sequence $\{x_{n}\}$ is also bounded.
\end{lemma}

\begin{lemma}\label{ecap}\cite[Proposition 2.2]{rei}
Let $f:E\to\mathbb{R}$ be a G\^{a}teaux differentiable and totally convex function, $x_{0}\in E$ and let $C$ be a nonempty, closed convex subset of $E$. Suppose that the sequence $\{x_{n}\}$ is bounded and any weak subsequential limit of $\{x_{n}\}$ belongs to $C$. If $D_{f}(x_{n},x_{0})\leq D_{f}(proj^{f}_{C}x_{0},x_{0})$ for any $n\in N$, then $\{x_{n}\}_{n=1}^{\infty}$ converges to $proj^{f}_{C}x_{0}$.
\end{lemma}

\begin{definition}\cite{rei}
Let $T:C\to C$ be a nonlinear mapping. The fixed points set of $T$ is denoted by $F(T)$, that is $F(T)=\{x\in C: Tx=x\}$. A mapping $T$ is said to be nonexpansive if $\|Tx-Ty\|\leq \|x-y\|$ for all $x,y\in C$. $T$ is said to be quasi-nonexpansive if $F(T)\neq \emptyset$ and $\|Tx-p\|\leq \|x-p\|,$ for all $x\in C$ and $p\in F(T)$. A point $p\in C$ is called an asymptotic fixed point of $T$ (see \cite{alb2}) if $C$ contains a sequence $\{x_{n}\}$ which converges weakly to $p$ such that $\lim_{n\to\infty}\|x_{n}-Tx_{n}\|=0$.
A point $p\in C$ is called a strong asymptotic fixed point of $T$ (see \cite{alb2}) if $C$ contains a sequence $\{x_{n}\}$ which converges strongly to $p$ such that $\lim_{n\to\infty}\|x_{n}-Tx_{n}\|=0$.
We denote the sets of asymptotic fixed points and strong asymptotic fixed points of $T$ by $\widehat{F}(T)$ and $\widetilde{F}(T)$, respectively.
\end{definition}

A mapping $T:C\to\text{int}(\text{dom} f)$ with $F(T)\neq\emptyset$ is called:
\begin{enumerate}
\item quasi-Bregman nonexpansive \cite{rei} with respect to $f$ if 
$$D_{f}(p,Tx)\leq D_{f}(p,x), \forall x\in C, p\in F(T).$$
\item
Bregman relatively nonexpansive \cite{rei} with respect to $f$ if,
$$ D_{f}(p,Tx)\leq D_{f}(p,x), \ \ \forall x\in C, p\in F(T), \ \ \ \text{and} \ \ \widehat{F}(T)=F(T).  $$
\item
Bregman strongly nonexpansive (see \cite{bru,rei}) with respect to $f$ and $\widehat{F}(T)$ if,
$$ D_{f}(p,Tx)\leq D_{f}(p,x), \ \ \forall x\in C, p\in \widehat{F}(T) $$
and, if whenever $\{x_{n}\}\subset C$ is bounded, $p\in \widehat{F}(T)$, and 
$$\lim_{z\to\infty}(D_{f}(p,x_{n})-D_{f}(p,Tx_{n}))=0,$$
it follows that 
$$\lim_{n\to\infty} D_{f}(x_{n},Tx_{n})=0.$$
\item
Bregman firmly nonexpansive (for short BFNE) with respect to $f$ if, for all $x,y\in C$,
$$\langle \nabla f(Tx)-\nabla f(Ty),Tx-Ty\rangle \leq \langle \nabla f(x)-\nabla f(y),Tx-Ty\rangle$$
equivalently,
\begin{equation}
D_{f}(Tx,Ty)+D_{f}(Ty,Tx)+D_{f}(Tx,x)+D_{f}(Ty,y)\leq D_{f}(Tx,y)+D_{f}(Ty,x).
\label{5}
\end{equation}
\item
a weak Bregman relatively nonexpansive mapping with $F(T)\neq\emptyset$ if $\widetilde{F}(T)=F(T)$ and
$$D_{f}(p,Tx)\leq D_{f}(p,x),  \ \ \forall x\in C, p\in F(T).$$
\end{enumerate}
The existence and approximation of Bregman firmly nonexpansive mappings was studied in \cite{rei2}. It is also known that if $T$ is Bregman firmly nonexpansive and $f$ is Legendre function which is bounded, uniformly Fr\'{e}chet differentiable and totally convex on bounded subset of $E$, then $F(T)=\widehat{F}(T)$ and $F(T)$ is closed and convex. It also follows that every Bregman firmly nonexpansive mapping is Bregman strongly nonexpansive with respect to $F(T)=\widehat{F}(T)$.

Let $C$ be a nonempty, closed  and convex subset of $E$. Let $f:E\to \mathbb{R}$ be a G\^{a}teaux differentiable and totally convex function. Let $x\in E$ it is known from \cite{but} that $z=proj_{C}^{f}(x)$ if and only if 
$$\langle \nabla f(x)-\nabla f(z),y-z\rangle \leq 0, \ \ \ \forall y\in C.$$
We also know the following:
\begin{equation}\label{9}
D_{f}(y,proj_{C}^{f}(x))+D_{f}(proj_{C}^{f}(x),x)\leq D_{f}(y,x),  \ \ \ \forall x\in E, y\in C.
\end{equation}

Let $f:E\to\mathbb{R}$ be a convex, Legendre and G\^{a}teaux differentiable function. Following \cite{alb} and \cite{cens}, we make use of the function $V_{f}:E\times E^{*}\to [0,\infty)$ associated with $f$, which is defined by
$$V_{f}(x,x^{*})=f(x)-\langle x^{*},x\rangle +f^{*}(x^{*}), \ \ \ \ \forall x\in E, x^{*}\in E^{*}.$$ 
Then $V_{f}$ is nonexpansive and $V_{f}(x,x^{*})=D_{f}(x,\nabla f^{*}(x^{*}))$ for all $x\in E$ and $x^{*}\in E^{*}$. Moreover, by the subdifferential inequality,
\begin{equation}\label{29}
V_{f}(x,x^{*})+\langle y^{*},\nabla f^{*}(x^{*})-x\rangle\leq V_{f}(x,x^{*}+y^{*})
\end{equation}
for all $x\in E$ and $x^{*},y^{*}\in E^{*}$ \cite{koh}. In addition, if $f:E\to (-\infty,+\infty]$ is a proper lower semicontinuous function, then $f^{*}:E^{*}\to(-\infty,+\infty]$ is a proper weak$^{*}$ lower semicontinuous and convex function (see \cite{mar}). Hence, $V_{f}$ is convex in the second variable. Thus, for all $z \in E$,
$$D_{f}\left(z,\nabla f^{*}\left(\sum_{i=1}^{N}t_{i}\nabla f(x_{i})\right)\right)\leq \sum_{i=1}^{N}t_{i}D_{f}(z,x_{i}),$$
where $\{x_{i}\}_{i=1}^{N}\subset E$ and $\{t_{i}\}_{i=1}^{N}\subset (0,1)$ with $\sum_{i=1}^{N}t_{i}=1$.
\begin{lemma}\label{agra}\cite{but}
Let $f\to(-\infty,+\infty]$ be G\^{a}teaux differentiable and totally convex on $\text{int}(\text{dom}f)$. Let $x\in \text{int}(\text{dom}f)$ and $C\subset \text{int}(\text{dom}f)$ be a nonempty, closed convex set. If $\hat{x}\in C$, then the following statements are equivalent:
\begin{enumerate}
\item
the vector $z$ is the Bregman projection of $x$ onto $C$ with respect to $f$;
\item
the vector $z$ is the unique solution of the variational inequality:
$$\langle \nabla f(x)-\nabla f(z),z-y\rangle\geq0, \ \ \forall y\in C;$$
\item
the vector $z$ is the unique solution of the inequality:
$$D_{f}(y,z)+D_{f}(z,x)\leq D_{f}(y,x), \ \ \forall y\in C.$$
\end{enumerate}
\end{lemma}

\begin{lemma}\label{333}\cite{rei3}
Let $C$ be a nonempty, closed and convex subset of $\text{int}(\text{dom} f)$ and $T:C\to C$ be a quasi-Bregman nonexpansive mappings with respect to $f$. Then $F(T)$ is closed and convex.
\end{lemma}

For solving the generalized mixed equilibrium problem, let us assume that the bifunction $\Theta :C\times C\to\mathbb{R}$ satisfies the following conditions:

($A_{1}$) $\Theta (x,x)=0$ for all $x\in C$;

($A_{2}$) $\Theta$ is monotone, i.e., $\Theta (x,y)+\Theta (y,x)\leq0$ for any $x,y\in C$;

($A_{3}$) for each $y\in C, x\mapsto \Theta (x,y)$ is upper-hemicontinuous, i.e., for each $x,y,z\in C$,
$$\limsup_{t\searrow 0}\Theta(tz+(1-t)x,y)\leq \Theta (x,y);$$

($A_{4}$) for each $x\in C, y\mapsto \Theta (x,y)$ is convex and lower semicontinuous.

\begin{definition}\cite{darv}
Let $C$ be a nonempty, closed and convex subsets of a real reflexive Banach space and let $\varphi$ be a lower semicontinuous and convex functional from $C$ to $\mathbb{R}$ and $\Psi:C\to E^{*}$ be a continuous monotone mapping. Let $\Theta: C\times C\to\mathbb{R}$ be a bifunctional satisfying ($A_{1}$)-($A_{4}$). The \textit{mixed resolvent} of $\Theta$ is the operator $Res_{\Theta,\varphi,\Psi}^{f}:E\to 2^{C}$ 
\begin{eqnarray}
Res_{\Theta,\varphi,\Psi}^{f}(x)&=&\{z\in C: \Theta (z,y)+\varphi(y)+\langle \Psi x,y-z\rangle+\langle \nabla f(z)-\nabla f(x),y-z\rangle\nonumber\\
&& \geq \varphi(z), \ \ \forall y\in C\}.\label{nres}
\end{eqnarray}
\end{definition}

\begin{lemma}\label{nv}\cite{darv}
Let $f:E\to(-\infty,+\infty]$ be a coercive Legendre function. Let $C$ be a closed and convex subset of $E$. If the bifunction $\Theta: C\times C\to\mathbb{R}$ satisfies conditions ($A_{1})$-($A_{4}$), then 
\begin{enumerate}

\item $Res_{\Theta,\varphi,\Psi}^{f}$ is single-valued and $\text{dom}(Res_{\Theta,\varphi,\Psi}^{f})=E$;

\item $Res_{\Theta,\varphi,\Psi}^{f}$ is a BFNE operator;

\item $F\left(Res_{\Theta,\varphi,\Psi}^{f}\right)=GMEP(\Theta)$;

\item $GMEP(\Theta)$ is closed and convex;

\item $D_{f}\left(p, Res_{\Theta,\varphi,\Psi}^{f}(x)\right)+D_{f}\left(Res_{\Theta,\varphi,\Psi}^{f}(x),x\right)\leq D_{f}(p,x), \ \forall p\in F\left(Res_{\Theta,\varphi,\Psi}^{f}\right),\break x\in E$.
\end{enumerate}
\end{lemma}

\begin{lemma}\cite[Proposition 5]{kas}
Let $f:E\to\mathbb{R}$ be a Legendre function such that $\nabla f^{*}$ is bounded on bounded subset of $int dom f^{*}$. For $x\in E$, if $\{D_{f}(x,x_{n})\}$ is bounded, then the sequence $\{x_{n}\}$ is bounded.
\end{lemma}

\section{Main result}
In this section, we prove our main theorem.
\begin{theorem}\label{mt}
Let $E$ be a real reflexive Banach space, $C$ be a nonempty, closed and convex subset of $E$. Let $f:E\to \mathbb{R}$ be a coercive Legendre function which is bounded, uniformly Fr\'{e}chet differentiable and totally convex on bounded subsets of $E$, and $\nabla f^{*}$ be bounded on bounded subset of $E^{*}$. Let $T:C\to C$ be a weak Bregman relatively nonexpansive mapping, $\Theta:C\times C\to\mathbb{R}$ satisfying conditions ($A_{1}$)-($A_{4}$), $\varphi:C\to \mathbb{R}$ is real-valued convex function and $\Psi: C\to E^{*}$ is continuous monotone mapping. Assume that $F(T)\cap GMEP(\Theta)$ is nonempty and bounded. Let $\{x_{n}\}$ be a sequence generated by
\begin{align}
&z_{n}=\nabla f^{*}(\beta_{n}\nabla f(T(x_{n}))+(1-\beta_{n})\nabla f(x_{n})), \nonumber\\
&y_{n}=\nabla f^{*}(\alpha_{n} \nabla f(x_{0})+(1-\alpha_{n})\nabla f(z_{n})),\nonumber\\
&u_{n}=Res_{\Theta,\varphi,\Psi}^{f}(y_{n}),\nonumber\\
&C_{n+1}=\{z\in C_{n}: D_{f}(z,u_{n})\leq \alpha_{n}D_{f}(z,x_{0})+(1-\alpha_{n})D_{f}(z,x_{n})\}, \nonumber\\
&Q_{n+1}=\{z\in Q_{n}: \langle\nabla f(x_{0})-\nabla f(x_{n}),z-x_{n}\rangle\leq 0\},\nonumber\\
&x_{n+1}= proj^{f}_{C_{n+1}\cap Q_{n+1}}x_{0}, \ \forall n\geq 0,\label{main}
\end{align}
where $\{\alpha_{n}\}, \{\beta_{n}\}\subset (0,1)$ satisfying $\lim_{n\to\infty}\alpha_{n}=0$ and $\liminf_{n\to\infty}(1-\alpha_{n})\beta_{n}>0$. Let $x_{0}\in C$ chosen arbitrarily, $Q_{0}=C$ and $C_{0}=\{z\in C: D_{f}(z,u_{0})\leq D_{f}(z,x_{0})\}$. Then,
 $\{x_{n}\}$ converges strongly to $proj^{f}_{F(T)\cap GMEP(\Theta)}x_{0}$.
\end{theorem}
\begin{proof}

We prove our theorem by several claims:

\begin{claim}\label{pro1} The sequence $\{x_n\}$ in \eqref{main} is well defined.
\end{claim}
We note from Lemmas \ref{333} and \ref{nv} that $F(T)$ and $GMEP(\Theta)$ are closed and convex.

First, we show that $C_n$ and $Q_n$ are closed and convex subsets of $E$. It is clear that $C_0$ and $Q_0$ are closed and convex subsets. Suppose that $C_n$ and $Q_n$ are closed and convex subsets of $E$ for some $n\geq 0$. We rewrite the set $C_{n+1}$ in the following form
\begin{align*}
C_{n+1}&=C_n\cap \{z\in E:\ D_{f}(z,u_{n})\leq \alpha_{n}D_{f}(z,x_{0})+(1-\alpha_{n})D_{f}(z,x_{n})\}\\
&=C_n\cap \{z\in E:\ \langle \alpha_n \nabla f(x_0)+(1-\alpha_n)\nabla f(x_n)-f(u_n),z\rangle \leq \alpha_n \langle \nabla f(x_0),x_0\rangle\\
&\quad +(1-\alpha_n)\langle \nabla f(x_n),x_n\rangle -\alpha_n f(x_0)-(1-\alpha_n)f(x_n) +f(u_n)-\langle \nabla f(u_n),u_n\rangle.
\end{align*}
Thus, $C_{n+1}$ is closed and convex subset of $E$.

Next, it follows from
\begin{align*}
Q_{n+1}&=Q_n\cap \{z\in E:\ \langle \nabla f(x_0)-\nabla f(x_n),z\rangle \leq \langle \nabla f(x_0)-\nabla f(x_n),x_n\rangle,
\end{align*}
that $Q_{n+1}$ is also closed and convex subset of $E$.

Now, in order to finish the proof of this claim, we will prove that $F(T)\cap GMEP(\Theta)\subset C_n\cap Q_n$ for all $n\geq 0$. Indeed, obviously $F(T)\cap GMEP(\Theta)\subset C_0\cap Q_0$. We suppose that $F(T)\cap GMEP(\Theta)\subset C_n\cap Q_n$ for some $n\geq 0$.\\
Let $p\in F(T)\cap GMEP(\Theta)$, from (\ref{main}) and Lemma \ref{nv}, we have 
\begin{align}
D_{f}(p,u_{n})&=D_{f}(p,Res_{\Theta,\varphi,\Psi}^{f}(y_{n}))\notag\\
&\label{eqt33}\leq D_f(p,y_n)-D_{f}(Res_{\Theta,\varphi,\Psi}^{f}(y_{n}),y_n)\notag.\\
\end{align}
Next, we have 
\begin{align}
D_f(p,y_n)&=D_{f}(p,\nabla f^{*}\left(\alpha_{n}\nabla f(x_{0})+(1-\alpha_{n})\nabla f(z_{n}))\right)\notag\\
&\label{eqt34}\leq \alpha_{n}D_{f}(p,x_{0})+(1-\alpha_{n})D_{f}(p,z_{n}).
\end{align}
We now estimate $D_{f}(p,z_{n})$, it follows form \eqref{main}, and the property of $T$ that
\begin{align}
D_f(p,z_n)&=D_{f}(p,\beta_{n}\nabla f(T(x_{n}))+(1-\beta_{n})\nabla f(x_{n}))\notag\\
&\leq \beta_{n}D_{f}(p,T{x_n})+(1-\beta_{n})D_{f}(p,x_{n})\notag\\
&\label{eqt35}\leq D_{f}(p,x_{n}).
\end{align}
From \eqref{eqt33}--\eqref{eqt35}, we get that
$$D_f(p,u_n)\leq \alpha_n D_f(p,x_0)+(1-\alpha_n)D_f(p,x_n).$$
This implies that $p\in C_{n+1}$ and hence $F(T)\cap GMEP(\Theta)\subset C_{n+1}$.

Since $x_{n}=proj^{f}_{C_{n}\cap Q_{n}}(x_{0})$, it follows from Lemma \ref{agra} that
$$\langle \nabla f(x_{0})-\nabla f(x_{n}),x_{n}-v\rangle\geq 0, \ \forall v\in C_{n}\cap Q_{n}.$$
Thus, from $p\in F(T)\cap GMEP(\Theta)\subset C_n\cap Q_n$, we obtain that
$$\langle \nabla f(x_{0})-\nabla f(x_{n}),x_{n}-p\rangle\geq 0,$$
that is, $p\in Q_{n+1}$ and hence $F(T)\cap GMEP(\Theta)\subset Q_{n+1}$. So, we deduce that $F(T)\cap GMEP(\Theta)\subset C_{n+1}\cap Q_{n+1}$. By mathematical induction, we get that $F(T)\cap GMEP(\Theta)\subset C_{n}\cap Q_{n}$ for all $n\geq 0$.

Thus, $C_n\cap Q_n$ is nonempty, closed and convex subset of $E$ for all $n\geq 0$ and hence the sequence $\{x_n\}$ is well defined.

\begin{claim}\label{pro2}
In \eqref{main}, the sequence $\{x_n\}$ is bounded.
\end{claim}
Since $\langle \nabla f(x_{0})-\nabla f(x_{n}),v-x_{n}\rangle\leq0$ for all $v\in Q_{n+1}$, it follows from Lemma \ref{agra} that $x_{n}=proj^{f}_{Q_{n+1}}x_{0}$ and by $x_{n+1}=proj^{f}_{C_{n+1}\cap Q_{n+1}}x_{0}\in Q_{n+1}$, we have
\begin{equation}\label{3.2f}
D_{f}(x_{n},x_{0})\leq D_{f}(x_{n+1},x_{0}).
\end{equation}
Let $p\in F(T)\cap GMEP(\Theta)\in Q_{n+1}$. It follows from Lemma \ref{agra} that
$$D_{f}(p,proj^{f}_{Q_{n+1}}x_{0})+D_{f}(proj^{f}_{Q_{n+1}}x_{0},x_{0})\leq D_{f}(p,x_{0})$$
and so
$$D_{f}(x_{n},x_{0})\leq D_{f}(p,x_{0})-D_{f}(p,x_{n})\leq D_{f}(p,x_{0}).$$
Therefore, $\{D_{f}(x_{n},x_{0})\}$ is bounded. By Lemma \ref{aust} $\{x_{n}\}$ is bounded and so are $\{T(x_{n})\}, \{y_{n}\}, \{z_{n}\}$.

\begin{claim}\label{pro3}
In \eqref{main}, the sequence $\{x_{n}\}$ is a Cauchy sequence.
\end{claim}

By the proof of Claim \ref{pro2}, we know that $\{D_{f}(x_{n},x_{0})\}$ is bounded. It follows from (\ref{3.2f}) that $\lim_{n\to\infty}D_{f}(x_{n},x_{0})$ exists. From $x_{m}\in Q_{m}\subseteq Q_{n+1}$ for all $m>n$ and Lemma \ref{agra}, we have
$$D_{f}(x_{m},proj_{Q_{n+1}}x_{0})+D_{f}(proj^{f}_{Q_{n+1}}x_{0},x_{0})\leq D_{f}(x_{m},x_{0})$$
and hence $D_{f}(x_{m},x_{n})\leq D_{f}(x_{m},x_{0})-D_{f}(x_{n},x_{0})$. Therefore, we have
\begin{equation}\label{3.3f}
\lim_{n\to\infty}D_{f}(x_{m},x_{n})\leq \lim_{n,m\to\infty}(D_{f}(x_{m},x_{0})-D_{f}(x_{n},x_{0}))=0.
\end{equation}
Since $f$ is totally convex on bounded subsets of $E$, by Definition \ref{deff}, Lemma  \ref{2.2.f} and (\ref{3.3f}) we obtain 
\begin{equation}\label{3.4f}
\lim_{n\to\infty}\|x_{m}-x_{n}\|=0.
\end{equation}
Thus $\{x_{n}\}$ is a Cauchy sequence and so $\lim_{n\to\infty}\|x_{n+1}-x_{n}\|=0$.

Now, we prove that the sequence $\{x_n\}$ generated by \eqref{main} converges strongly to $x^\dagger =proj^{f}_{F(T)\cap GMEP(\Theta)}x_{0}$.

From the proof of Claim \ref{pro2}, the sequence $\{x_{n}\}$ is a Cauchy sequence. Without of loss of generality, let $x_{n}\to q\in C$. Since $f$ is uniformly Fr\'{e}chet differentiable on bounded subsets of $E$. It follows from Lemma \ref{2.2.f} that $\nabla f$ is norm-to-norm uniformly continuous on bounded subsets of $E$. Hence, by $\|x_{n+1}-x_n\|\to 0$, we have
\begin{equation}
\lim_{n\to\infty}\|\nabla f(x_{n+1})-\nabla f(x_{n})\|=0.
\end{equation}
Since $x_{n+1}\in C_{n+1}\subset C_{n}$, we have
$$D_{f}(x_{n+1},u_{n})\leq \alpha_{n}D_{f}(x_{n+1},x_{0})+(1-\alpha_{n})D_{f}(x_{n+1},x_{n}).$$
It follows from $\lim_{n\to\infty}\alpha_{n}=0$ and $\lim_{n\to\infty}D_{f}(x_{n+1},x_{n})=0$ that $\{D_{f}(x_{n+1},u_{n})\}$ is bounded and
$$\lim_{n\to\infty}D_{f}(x_{n+1},u_{n})=0.$$
By Lemma \ref{lem6}, we obtain
\begin{equation}\label{mahnaz1}
\lim_{n\to\infty}\|x_{n+1}-u_{n}\|=0.
\end{equation}
So, 
\begin{equation}\label{3.10ma}
\lim_{n\to\infty}\|\nabla f(x_{n+1})-\nabla f(u_{n})\|=0.
\end{equation}
Taking into account that $\|x_{n}-u_{n}\|\leq \|x_{n}-x_{n+1}\|+\|x_{n+1}-u_{n}\|$, we obtain
$$\lim_{n\to\infty}\|x_{n}-u_{n}\|=0$$
so, $u_{n}\to q$ as $n\to\infty$.\\
By Bregman distance we have
\begin{eqnarray*}
&&D_{f}(p,x_{n+1})-D_{f}(p,u_{n})\\
&&=f(p)-f(x_{n+1})-\langle \nabla f(x_{n+1}),p-x_{n+1}\rangle -f(p)+f(u_{n})+\langle \nabla f(u_{n}),p-u_{n}\rangle\\
&&=f(u_{n})-f(x_{n+1})+\langle \nabla f(u_{n}),p-u_{n}\rangle-\langle \nabla f(x_{n+1}),p-x_{n+1}\rangle\\
&&=f(u_{n})-f(x_{n+1})+\langle \nabla f(u_{n}),x_{n+1}-u_{n}\rangle+\langle \nabla f(u_{n})-\nabla f(x_{n+1}),p-x_{n+1}\rangle,
\end{eqnarray*}
for each $p\in F(T)$. By (\ref{mahnaz1})-(\ref{3.10ma}), we obtain
\begin{equation}\label{3.11}
\lim_{n\to\infty}(D_{f}(p,x_{n+1})-D_{f}(p,u_{n}))=0.
\end{equation}
On the other hand, for any $p\in F(T)\cap GMEP(\Theta)$ by Lemma \ref{nv}, we have
\begin{align}
&D_{f}(u_{n},y_{n})\leq D_{f}(p,y_{n})-D_{f}(p,u_{n})\nonumber\\
&=D_{f}(p,\nabla f^{*}(\alpha_{n}\nabla f(x_{0})+(1-\alpha_{n})\nabla f(z_{n})))-D_{f}(p,u_{n})\nonumber\\
&\leq \alpha_{n}D_{f}(p,x_{0})+(1-\alpha_{n})D_{f}(p,\nabla f^{*}(\beta_{n}\nabla f(T(x_{n}))\nonumber\\
&+(1-\beta_{n})\nabla f(x_{n})))-D_{f}(p,u_{n})\nonumber\\
&\leq \alpha_{n}D_{f}(p,x_{0})+(1-\alpha_{n})D_{f}(p,x_{n})-D_{f}(p,u_{n})\nonumber\\
&=\alpha_{n}(D_{f}(p,x_{0})-D_{f}(p,x_{n}))+D_{f}(p,x_{n})-D_{f}(p,u_{n}).
&\nonumber\\
\label{mahnaz3}
\end{align}
So, by (\ref{3.11}) and (\ref{mahnaz3}) we have $D_{f}(u_{n},y_{n})=0$ and $D_{f}(p,y_{n})-D_{f}(p,u_{n})\to 0$ as $n\to\infty$. Moreover, $\lim_{n\to\infty}\|u_{n}-y_{n}\|=0$ and thus $\lim_{n\to\infty}\|\nabla f(u_{n})-\nabla f(y_{n})\|=0$.
Since $u_{n}\to q$ as $n\to\infty$, we have $y_{n}\to q$ as $n\to\infty$.

Here, we prove that $q\in GMEP(\Theta)$. 
It follows form \eqref{eqt33} and $$D_{f}(p,y_{n})-D_{f}(p,u_{n})\to 0$$ that 
\begin{equation}
\label{eqt314}
D_f(u_{n},y_n)\to 0.
\end{equation}
Moreover, from \eqref{eqt314}, we also have that
\begin{equation}
\label{eqt318}
\|\nabla f(u_{n})-\nabla f(y_n)\|\to 0.
\end{equation}
Also, consider that $u_{n}=Res_{\Theta,\varphi,\Psi}^{f}(y_{n})$, so we have
$$\Theta(u_{n},y)+\langle \Psi y_{n},y-u_{n}\rangle+\varphi(y)+\langle \nabla u_{n}-\nabla y_{n},y-u_{n}\rangle\geq \varphi(u_{n}),$$ 
for all $y\in C.$\\
From ($A_{2}$), we have
$$\Theta(y,u_{n})\leq -\Theta(u_{n},y)\leq \langle \Psi y_{n},y-u_{n}\rangle+ \varphi(y)-\varphi(u_{n})+\langle \nabla u_{n}-\nabla y_{n},y-u_{n}\rangle,$$
for all $y\in C.$\\
Hence, 
$$\Theta(y,u_{n})\leq \langle \Psi y_{n},y-u_{n}\rangle+\varphi(y)- \varphi(u_{n})+\langle \nabla u_{n}-\nabla y_{n},y-u_{n}\rangle,$$
for $y\in C.$\\
Since $u_{n}\to q$ and \eqref{eqt318}, from continuity of $\Psi$ and weak lower semicontinuity of $\varphi$ and $\Theta (\cdot,\cdot)$ in the second variable $y$, we also have
$$\Theta (y,q)+\langle \Psi q,q-y\rangle+\varphi(q)-\varphi(y)\leq0,$$
for all $y\in C.$\\
For $t$ with $0\leq t\leq 1$ and $y\in C$, let $y_{t}=ty+(1-t)q$. Since $y\in C$ and $q\in C$ we have $y_{t}\in C$ and hence $\Theta(y_{t},q)+\langle \Psi q,q-y_{t}\rangle+\varphi(q)-\varphi(y_{t})\leq 0$. So, we have
\begin{eqnarray*}
0&=&\Theta(y_{t},y_{t})+\langle \Psi q,y_{t}-y_{t}\rangle+\varphi(y_{t})-\varphi(y_{t})\\
&\leq& t\Theta(y_{t},y)+(1-t)\Theta(y_{t},q)+t\langle \Psi q,y-y_{t}\rangle+(1-t)\langle \Psi q,q-y_{t}\rangle\\&&+t\varphi(y)+(1-t)\varphi(q)-\varphi(y_{t})\\
&\leq& t[\Theta(y_{t},y)+\langle \Psi q,y-y_{t}\rangle+\varphi(y)-\varphi(y_{t})].
\end{eqnarray*}
 Therefore, $\Theta(y_{t},y)+\langle \Psi q,y-y_{t}\rangle+\varphi(y)-\varphi(y_{t})\geq0$. Then, we have 
 $$\Theta(q,y)+\langle \Psi q,y-q\rangle+\varphi(y)-\varphi(q)\geq0,$$
for all $y\in C.$ Hence, we have $q\in GMEP(\Theta)$.\\
 
Now, we  prove that $q\in F(T)$. Note that
\begin{eqnarray*}
&&\|\nabla f(x_{n})-\nabla f(y_{n})\|=\|\nabla f(x_{n})-\nabla f(\nabla f^{*}(\alpha_{n}\nabla f(x_{0})+(1-\alpha_{n})\nabla f(z_{n})))\|\\
&&=\|\nabla f(x_{n})-(\alpha_{n}\nabla f(x_{0})+(1-\alpha_{n})\nabla f(z_{n}))\|\\
&&=\|\alpha_{n}(\nabla f(x_{n})-\nabla f(x_{0}))+(1-\alpha_{n})(\nabla f(x_{n})-\nabla f(z_{n}))\|\\
&&=\|\alpha_{n}(\nabla f(x_{n})-\nabla f(x_{0}))+(1-\alpha_{n})(\nabla f(x_{n})-\nabla f (\nabla f^{*}(\beta_{n}\nabla f(T(x_{n}))\\\
&&+(1-\beta_{n})\nabla f(x_{n}))))\|\\
&&=\|\alpha_{n}(\nabla f(x_{n})-\nabla f(x_{0}))+(1-\alpha_{n})\beta_{n}(\nabla f(x_{n})-\nabla f(T(x_{n})))\|\\
&&\geq (1-\alpha_{n})\beta_{n}\|\nabla f(x_{n})-\nabla f(T(x_{n}))\|-\alpha_{n}\|\nabla f(x_{n})-\nabla f(x_{0})\|.
\end{eqnarray*}
This implies that
\begin{equation}\label{mahnaz4}
(1-\alpha_{n})\beta_{n}\|\nabla f(x_{n})-\nabla f(T(x_{n}))\|\leq \alpha_{n}\|\nabla f(x_{n})-\nabla f(x_{0})\|+\|\nabla f(x_{n})-\nabla f(y_{n})\|.
\end{equation}
Letting $n\to\infty$ in the above inequality, it follows from $\liminf_{n\to\infty}(1-\alpha_{n})\beta_{n}>0$ and $\lim_{n\to\infty}\alpha_{n}=0$ that
$$\lim_{n\to\infty}\|\nabla f(x_{n})-\nabla f(T(x_{n}))\|=0.$$
So, we have $\lim_{n\to\infty}\|x_{n}-T(x_{n})\|=0$. This together with $x_{n}\to q$ implies that $q\in \widetilde{F}(T)$. Since $\widetilde{F}(T)=F(T)$, we have $q\in F(T)\cap GMEP(\Theta)$. Therefore, the sequence $\{x_{n}\}$ converges strongly to a point $q\in F(T)\cap GMEP(\Theta)$.

Finally, we prove that $q=x^\dagger =proj^{f}_{F(T)\cap GMEP(\Theta)}(x_{0})$. Since 
$$x^\dagger =proj^{f}_{F(T)\cap GMEP(\Theta)}(x_{0})\in F(T)\cap GMEP(\Theta)$$ it follows from $x_{n+1}=proj^{f}_{C_{n+1}\cap Q_{n+1}}x_{0}$ and $x^\dagger \in F(T)\cap GMEP(\Theta)\subset C_{n+1}\cap Q_{n+1}$ that
$$D_{f}(x_{n+1},x_{0})\leq D_{f}(x^\dagger ,x_{0}).$$
Hence by Lemma \ref{ecap}, we have $x_{n}\to x^\dagger $ as $n\to\infty$. Thus $q=x^\dagger $. Therefore, the sequence $\{x_{n}\}$ converges strongly to the point $$x^\dagger =proj^{f}_{F(T)\cap GMEP(\Theta)}x_{0}.$$ This completes the proof.
\end{proof}

Let $\varphi=\Psi=0$, then we have the result of \cite{aga} as follows:
\begin{corollary}
Let $E$ be a real reflexive Banach space, $C$ be a nonempty, closed and convex subset of $E$. Let $f:E\to \mathbb{R}$ be a coercive Legendre function which is bounded, uniformly Fr\'{e}chet differentiable and totally convex on bounded subsets of $E$, and $\nabla f^{*}$ be bounded on bounded subset of $E^{*}$. Let $T:C\to C$ be a weak Bregman relatively nonexpansive mapping, $\Theta:C\times C\to\mathbb{R}$ satisfying conditions ($A_{1}$)-($A_{4}$). Assume that $F(T)\cap EP(\Theta)$ is nonempty and bounded. Let $\{x_{n}\}$ be a sequence generated by
\begin{align}
&z_{n}=\nabla f^{*}(\beta_{n}\nabla f(T(x_{n}))+(1-\beta_{n})\nabla f(x_{n})), \nonumber\\
&y_{n}=\nabla f^{*}(\alpha_{n} \nabla f(x_{0})+(1-\alpha_{n})\nabla f(z_{n})),\nonumber\\
&u_{n}=Res_{\Theta}^{f}(y_{n}),\nonumber\\
&C_{n+1}=\{z\in C_{n}: D_{f}(z,u_{n})\leq \alpha_{n}D_{f}(z,x_{0})+(1-\alpha_{n})D_{f}(z,x_{n})\}, \nonumber\\
&Q_{n+1}=\{z\in Q_{n}: \langle\nabla f(x_{0})-\nabla f(x_{n}),z-x_{n}\rangle\leq 0\},\nonumber\\
&x_{n+1}= proj^{f}_{C_{n+1}\cap Q_{n+1}}x_{0}, \ \forall n\geq 0,
\end{align}
where $\{\alpha_{n}\}, \{\beta_{n}\}\subset (0,1)$ satisfying $\lim_{n\to\infty}\alpha_{n}=0$ and $\liminf_{n\to\infty}(1-\alpha_{n})\beta_{n}>0$. Let $x_{0}\in C$ chosen arbitrarily, $Q_{0}=C$ and $C_{0}=\{z\in C: D_{f}(z,u_{0})\leq D_{f}(z,x_{0})\}$. Then,
 $\{x_{n}\}$ converges strongly to $proj^{f}_{F(T)\cap EP(\Theta)}x_{0}$.
\end{corollary}

\section{Numerical example}
In this section, we present the example illustrating the behaviour of the iterative algorithm presented in this paper. Moreover, we compare the number of iterations of the sequences generated by iteration \eqref{e1.6} and iteration \eqref{main}.
\begin{example}
	Let $E=\mathbb{R}$, $C=[-\frac{3}{2},0)$, and  $f:\mathbb{R}\to\mathbb{R}$ be defined by $f(x)=\frac{1}{2}x^2$. Let $T:C\to C$ be defined by $Tx=\frac{2}{3}x$, and the bifunction $H:C\times C\to\mathbb{R}$ defined by $H(x,y)=x(y-x)$ see (\cite{aga}, Theorem 3.1). Let  $\Theta:C\times C\to \mathbb{R}$ such that $\Theta(x,y)=x(y-x)=H(x,y)$ for all $x,y\in C$, $\varphi:C\to \mathbb{R}$ be defined by $\varphi(x)=x^2$ for all $x\in C$ and $\Psi:C\to E^*$ such that $\Psi(x)=\sin(x)$ for all $x\in C$. Set $\{\alpha_n\}=\{\frac{1}{n+3}\}$ and $\{\beta_n\}=\{0.99-\frac{1}{n+2}\}$ for all $n\geq 0$. 
	
	We observe that $f$ is a coercive Legendre function which is bounded, uniformly Fr\'{e}chet differentiable and totally convex on bounded subsets of $\mathbb{R}$ and $\nabla f(x)=x$.  Since $f^*(x^*)=\sup\{\langle x^*,x\rangle-f(x):x\in \mathbb{R}\}$, we obtain that $f^*(u)=\frac{1}{2}u^2$ and $\nabla f^*(u)=u$. Further, we observe that $T$ is a weak Bregman relatively nonexpansive mapping with $\tilde{F}(T)=\{0\}=F(T)$. We also observe that $\Theta$, $H$ satisfy conditions $(A_1) - (A_4)$ and $\varphi$, $\Psi$ are a convex function and a continuous monotone mapping, respectively. Moreover, we have $GMEP(\Theta)=\{0\}=EP(H)$. Let $\{x_n\}$ be generated by the iterations \eqref{e1.6} and \eqref{main}. Then the sequence $\{x_n\}$ converges strongly to $0$, where $proj_{F(T)\cap EP(H)}(x_0)=0=proj_{F(T)\cap GMEP(\Theta)}(x_0)$. 
	The Algorithm \eqref{e1.6} and Algorithm \eqref{main} are checked by using the stopping criterion $||x_n-x_{n+1}||<10^{-10}$.
\end{example}

\begin{table}[h]
	\caption{The numerical results for different initial values $x_0$}
	\vspace{0.2 cm}
	\centering
	{\renewcommand{\arraystretch}{1.5}
	\begin{tabular}{ c c c }
		\hline
		Initial point&  \multicolumn{2}{ c }{\hspace{1 cm}Average iterations} \\
		\cline{2-3}
		$x_0$ & \hspace{1 cm}Algorithm \eqref{main} \hspace{1 cm}& \hspace{1 cm}Algorithm \eqref{e1.6} \\ 
		\hline
		-$\frac{1}{2}$ & \hspace{1 cm}1840206 & \hspace{1 cm}2001482  \\
		-1 & \hspace{1 cm}2921737 & \hspace{1 cm}3177798 \\
		$-\frac{3}{2}$ & \hspace{1 cm}3828937 & \hspace{1 cm}4164504  \\
		\hline
	\end{tabular}}
\end{table}
\begin{figure}[h]
	\caption{The numerical results for the generalized mixed equilibrium problem and the equilibrium problem}
	\label{fig:figure1}
	\includegraphics[height=8cm]{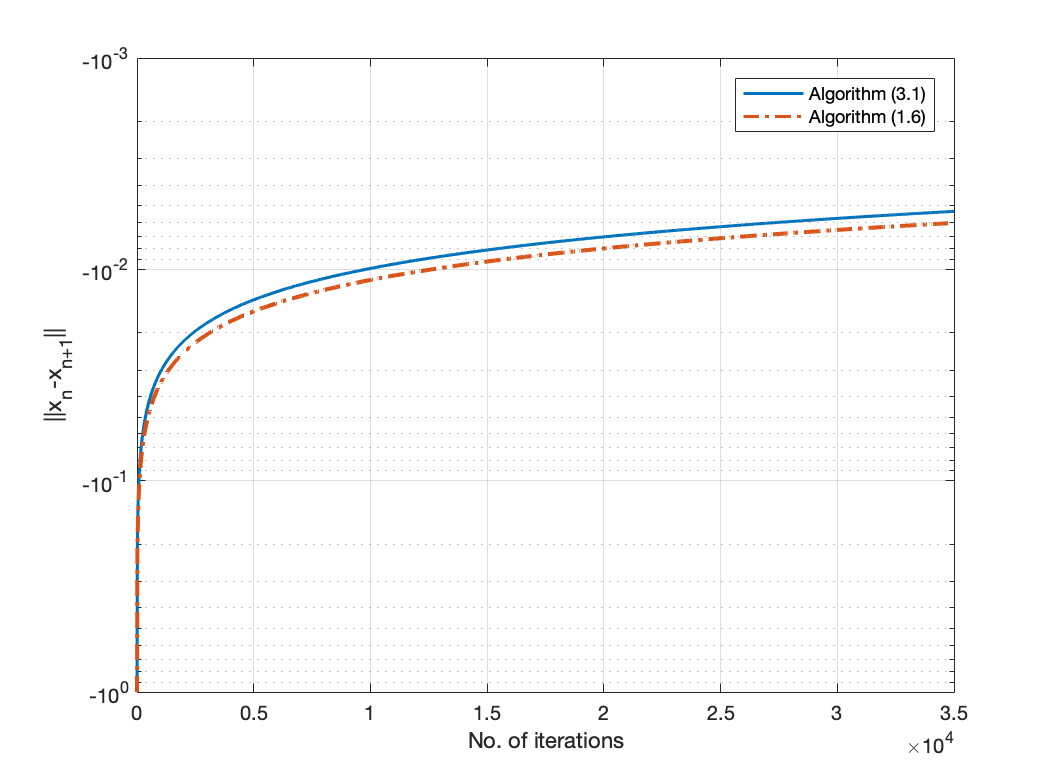}
\end{figure}

\section*{Acknowledgments}
The first author is supported by the Talented Young Scientist Program of Ministry of Science and Technology
of China (Iran-19-001).

\end{document}